\documentclass[12pt]{article}
\usepackage{amsmath}
\usepackage[cp1250]{inputenc}
\usepackage[T1]{fontenc}
\usepackage{amssymb, latexsym}
\usepackage{amsthm}
\usepackage{amscd}
\usepackage{makeidx}
\usepackage{amsfonts,mathrsfs}
\usepackage{topcapt}
\usepackage{hyperref}
\hypersetup{colorlinks=true, debug=true, linkcolor=blue,
hypertexnames=false}
\let\originalforall=\forall
\renewcommand{\forall}{\mathop{\vcenter{\hbox{\Large$\originalforall$}}}}

\let\originalexists=\exists
\renewcommand{\exists}{\mathop{\vcenter{\hbox{\Large$\originalexists$}}}}
\newtheorem{de}{Definition}
\newtheorem{fakt}[de]{Fact}

\newtheorem{tw}[de]{Theorem}

\newtheorem{lem}[de]{Lemma}
\newtheorem{cor}[de]{Corollary}

\newtheorem{rem}[de]{Remark}
\newtheorem{pro}[de]{Problem}
\date{}
\title{Non-separability of the Gelfand space of measure algebras}
\begin{document}
\author{Przemysław Ohrysko\thanks{The research of this author has been supported by National Science Centre, Poland grant no. 2014/15/N/ST1/02124}\\
Institute of Mathematics, Polish Academy of Sciences\\
00-956 Warszawa, Poland\\
E-mail: pohrysko@impan.pl \and
Michał Wojciechowski\\
Institute of Mathematics, Polish Academy of Sciences\\
00-956 Warszawa, Poland\\
E-mail: miwoj-impan@o2.pl \and
Colin C. Graham\\
Department of Mathematics, University of British Columbia\\
Mail address: Box 2031, Haines Junction YT Y0B 1L0 Canada\\
E-mail: c.c.graham@math.ubc.ca}
\maketitle
\renewcommand{\thefootnote}{}



\renewcommand{\thefootnote}{\arabic{footnote}}
\setcounter{footnote}{0}
\begin{abstract} We prove that there exists uncountably many pairwise disjoint open subsets of the Gelfand space of the measure algebra on any locally compact non-discrete abelian group which shows that this space is not separable (in fact, we prove this assertion for the ideal $M_{0}(G)$ consisting of measures with Fourier-Stieltjes transforms vanishing at infinity which is a stronger statement). As a corollary, we obtain that the spectras of elements in the algebra of measures cannot be recovered from the image of one countable subset of the Gelfand space under Gelfand transform, common for all elements in the algebra.
\end{abstract}
\section{Introduction}
We recall standard definitions and facts from commutative harmonic analysis and Banach algebra theory (see \cite{r},\cite{kan} and \cite{ż}). Let $G$ denote a locally compact abelian group with the dual group $\Gamma$ and let $M(G)$ be the algebra of all complex-valued Borel regular measures on $G$ equipped with a total variation norm and convolution product. With these operations $M(G)$ becomes a semisimple commutative, unital Banach algebra. The spectrum of a measure $\mu\in M(G)$ will be denoted by $\sigma(\mu)$. We define the Fourier-Stieljtes transform of $\mu\in M(G)$ as a mapping $\widehat{\mu}:\Gamma\rightarrow\mathbb{C}$ given by the formula
\begin{equation*}
\widehat{\mu}(\gamma)=\int_{G}\gamma(-x)d\mu(x)\text{ for $\gamma\in\Gamma$}.
\end{equation*}
The abbreviation $\triangle(M(G))$ will stand for the Gelfand space of $M(G)$ (the space of all maximal ideals in $M(G)$ or equivalently of all multiplicative linear functionals on $M(G)$). $\triangle(M(G))$ endowed with weak$^{\ast}$ topology is a compact Hausdorff topological space. Evaluation of the Fourier-Stieltjes transform of a measure at a point $\gamma\in\Gamma$ is a multiplicative linear functional which enables us to view the Fourier-Stieltjes transform of a measure as a restriction of its Gelfand transform to the set of uniqueness $\Gamma$. However, there are many different multiplicative linear functionals on $M(G)$ and $\Gamma$ is not dense in $M(G)$. The last fact is a reformulation of the well-known Wiener-Pitt phenomenon observed in \cite{wp} (for the first precise proof see \cite{schreider}, consult also \cite{wil},\cite{graham} and \cite{r}).
\\
Let $\mu\in M(G)$. We say that,
\begin{enumerate}
  \item $\mu$ is Hermitian if its Fourier-Stieltjes transform is real-valued,
  \item $\mu$ has independent powers if all distinct convolution powers of $\mu$ are mutually singular.
\end{enumerate}

We will mainly work with the ideal $M_{0}(G)$ of measures with Fourier-Stieltjes transforms vanishing at infinity and with its Gelfand space $\triangle(M_{0}(G))$ of all maximal modular ideals (multiplicative linear-functionals) which is locally compact but non-compact (because $M_{0}(G)$ is not unital). The spectrum of $\mu\in M_{0}(G)$ is defined as a spectrum of $\mu$ in a unitization of $M_{0}(G)$ which coincides with the spectrum of $\mu$ in $M(G)$ (this follows form the fact that all multiplicative linear functionals extend from the ideal to the whole algebra, see Lemma 2.2.5 in \cite{kan}) and hence will also be abbreviated by $\sigma(\mu)$.
\section{The case of a compact group}
We start with the definition of a dissociate set in a locally compact abelian group.

\begin{de}\label{dis}
A subset $\Theta$ of $\Gamma$ is called dissociate if every element $\omega\in\Gamma$ may be expressed in at most one way as a product
\begin{equation*}
\omega=\prod_{j=1}^{n}\theta_{j}^{\varepsilon_{j}},
\end{equation*}
where the $\theta_{i}$ are distinct elements of $\Theta$ and $\varepsilon_{j}$ is allowed to be $1$ or $-1$ when $\theta_{i}^{2}\neq 1$ but has to be equal $1$ if $\theta_{i}^{2}=1$. Every product of the above form will be called a word with letters $\theta_{1},\ldots,\theta_{n}$. The set of all words using letters from $\Theta$ will be denoted $\Omega(\Theta)$.
\end{de}
It is obvious from the definition that any subset of a dissociate set is a dissociate set. In the sequel we will need the following elementary lemma concerning dissociate sets.
\begin{lem}\label{ban}
Let $\Theta\subset\Gamma$ be dissociate. If $\Theta_{1},\Theta_{2}$ are subsets of $\Theta$ such that $\#\Theta_{1}\cap\Theta_{2}<\infty$ then $\#\left(\Omega(\Theta_{1})\cap\Omega(\Theta_{2})\right)<\infty$.
\end{lem}
\begin{proof}
Let us take $\theta\in\Omega(\Theta_{1})\cap\Omega(\Theta_{2})$. Then
\begin{equation*}
\theta=\prod_{j=1}^{n_{1}}\theta_{j}^{\varepsilon_{j}}=\prod_{k=1}^{n_{2}}\left(\theta_{k}'\right)^{\varepsilon'_{k}}
\end{equation*}
where $\theta_{j}\in\Theta_{1}$, $\theta_{k}'\in\Theta_{2}$ and the restrictions imposed on $\varepsilon_{j},\varepsilon_{k}'$ are the same as in Definition $\ref{dis}$. Now, let us observe that $\Theta_{1}\cup\Theta_{2}\subset\Theta$ is also a dissociate set. So $n_{1}=n_{2}$, $\theta_{j}=\theta'_{j}$ and $\varepsilon_{j}=\varepsilon'_{j}$ which gives $\theta\in\Omega(\Theta_{1}\cap\Theta_{2})$ and of course $\Omega(\Theta_{1}\cap\Theta_{2})$ is finite since $\#\Theta_{1}\cap\Theta_{2}<\infty$.
\end{proof}
From now on $G$ denotes a compact abelian group with the discrete dual group $\Gamma$. We are going to use Riesz products and hence we recall briefly their construction in the spirit of the introduction to Chapter 7 of \cite{grmc}. Let $\Theta\subset\Gamma$ be dissociate and let $a:\Theta\rightarrow\mathbb{C}$ be a function satisfying $-1<a(\theta)<1$ for $\theta\in\Theta$ with $\theta^{2}=1$ and $0\leq |a(\theta)|\leq\frac{1}{2}$ for $\theta\in\Theta$ with $\theta^{2}\neq 1$. We also define trigonometric polynomials $q_{\theta}$ as follows
\begin{equation*}
q_{\theta}=1+a(\theta)\theta\text{ if }\theta^{2}=1\text{ and }q_{\theta}=1+a(\theta)\theta+\overline{a(\theta)\theta}\text{ if }\theta^{2}\neq 1.
\end{equation*}
Our restrictions on the function $a$ implies that $q_{\theta}\geq 0$ for every $\theta\in\Theta$. For each finite subset $\Phi$ of $\Theta$ we define the partial product
\begin{equation*}
P_{\Phi}=\prod_{\theta\in\Phi}q_{\theta}.
\end{equation*}
We easily see that $P_{\Phi}$ is a non-negative element of $L^{1}(G)$ with norm one. Since $\Theta$ is dissociate, the formula for Fourier coefficients of $P_{\Phi}$ is given as follows ($\Phi=\{\theta_{1},\ldots,\theta_{n}\}$)
\begin{equation*}
\widehat{P_{\theta}}(\omega)=\begin{cases} 1\mbox{ if }\omega=1 \\
                                      \prod_{i=1}^{n}a(\theta_{i})^{(\varepsilon)}\mbox{ if }\omega=\prod_{j=1}^{n}\theta_{j}^{\varepsilon_{j}} \\
                                      0\mbox{ otherwise }
                             \end{cases}
\end{equation*}
Here $a(\theta_{i})^{(\varepsilon_{i})}=a(\theta_{i})$ if $\varepsilon_{i}=1$ and $a(\theta_{i})^{(\varepsilon_{i})}=\overline{a(\theta_{i})}$ if $\varepsilon_{i}=-1$. By this properties it is clear that the weak$^{\ast}$ limit of $\{P_{\Phi}\}$ (over increasing finite subsets $\Phi$ of $\Theta$) exists and is a probability measure which will be called the Riesz product based on $\Theta$ and $a$ and will be denoted by $R(a,\Theta)$. Its Fourier-Stieltjes coefficients are given by the similar formula to the one presented above (for simplicity we write $\mu:= R(a,\Theta)$)
\begin{equation*}
\widehat{\mu}(\omega)=\begin{cases}   1\mbox{ if }\omega=1, \\
                                      \prod_{i=1}^{n}a(\theta_{i})^{(\varepsilon)}\mbox{ if }\omega=\prod_{j=1}^{n}\theta_{j}^{\varepsilon_{j}}, \\
                                      0\mbox{ otherwise. }
                      \end{cases}
\end{equation*}
The following theorem (see Corollary 7.2.3 in \cite{grmc} and \cite{bm}) includes the easily verifiable condition on the property of having independent powers for Riesz products.
\begin{tw}\label{rip}
Let $G$ be an infinite compact abelian group and let $\mu:=R(a,\Theta)$ be a Riesz product based on $a$ and $\Theta$. Then $\mu$ has independent powers if and only if
\begin{equation*}
\sum\left\{|a(\theta)|^{2n}:|a(\theta)|<\frac{1}{2}\right\}+\sum\left\{1-|a(\theta):|a(\theta)|>\frac{1}{2}\right\}=\infty.
\end{equation*}
\end{tw}
It is an exercise in calculating spectra in Banach algebras that for a probability measure $\mu$ with independent powers we have $\{z\in\mathbb{C}:|z|=1\}\subset\sigma(\mu)$ (consider the subalgebra $A$ of $M(G)$ generated by $\mu$ and the unit and use the fact $\partial\sigma_{A}(\mu)\subset \partial\sigma(\mu)$). However, much stronger result is true (see Theorem 6.1.1. in \cite{grmc} and \cite{bbm}).
\begin{tw}\label{ips}
If $\mu$ is a Hermitian and probability measure on a locally compact abelian group $G$ with independent powers then $\sigma(\mu)=\{z\in\mathbb{C}:|z|\leq 1\}$.
\end{tw}
Following M. Zafran (see \cite{z}) we introduce the definition of measures with a natural spectrum.
\begin{de}
Let $\mu\in M(G)$. We say that $\mu$ has a natural spectrum if
\begin{equation*}
\sigma(\mu)=\overline{\widehat{\mu}(\Gamma)}.
\end{equation*}
\end{de}
From the formula of Fourier-Stietljes coefficients we see that if $a$ is a real-valued function satisfying the assumptions of Theorem $\ref{rip}$ then every Riesz product $R(a,\Theta)$ is a Hermitian measure probability measure with independent powers which, by Theorem $\ref{ips}$, gives
\begin{equation*}
\sigma(R(a,\Theta))=\{z\in\mathbb{C}:|z|\leq 1\}.
\end{equation*}
In particular, since in this case the Fourier-Stieltjes transform of $R(a,\Theta)$ is real-valued, the Riesz product $R(a,\Theta)$ does not have a natural spectrum (for the same conclusion restricted to $G=\mathbb{T}$ see Theorem 3.9 in \cite{z}).
\\
We are going to use the main theorem from the Zafran's paper (see Theorem 3.2 in \cite{z}).
\begin{tw}\label{zz}
Let $G$ be a compact abelian group. Let
\begin{equation*}
\mathscr{C}:=\{\mu\in M_{0}(G):\sigma(\mu)=\overline{\widehat{\mu}(\Gamma)}=\widehat{\mu}(\Gamma)\cup\{0\}\}.
\end{equation*}
Then the following hold true:
\begin{enumerate}
  \item If $h\in\triangle(M_{0}(G))\setminus\Gamma$ then $h(\mu)=0$ for $\mu\in\mathscr{C}$.
  \item $\mathscr{C}$ is a closed ideal in $M_{0}(G)$.
  \item $\triangle(\mathscr{C})=\Gamma$.
\end{enumerate}
\end{tw}
The last ingredient is an observation of Sierpiński.
\begin{fakt}
There exists uncountably many infinite subsets of natural numbers such that the intersection of each two is finite.
\end{fakt}
We are ready now to state and proof our main theorem
\begin{tw}\label{glo}
Let $G$ be a compact abelian group. Then $\triangle(M_{0}(G))$ contains continuum pairwise disjoint open subsets. In particular, $\triangle(M_{0}(G))$ is not separable.
\end{tw}
\begin{proof}
Let $\Theta$ be an infinite countable dissociate subset of $\Gamma$ enumerated by natural numbers. Then, by the observation of Sierpiński it contains uncountably many infinite subsets $\{\Theta_{\alpha}\}_{\alpha\in\mathbb{R}}$ such that the intersection of each two is finite. For $\alpha\in\mathbb{R}$ we form an auxiliary function $b_{\alpha}:\Theta_{\alpha}\rightarrow\mathbb{N}$ by the following recipe: $b_{\alpha}(\theta)$ is equal to the index of $\theta$ in $\Theta_{\alpha}$ (for example, if $\Theta_{\alpha}=\{\theta_{3^{k}}:k\in\mathbb{N}\}$ then $b_{\alpha}(\theta_{3^{k}})=k$). Also, for every $\alpha\in\mathbb{R}$, let us define a function $a_{\alpha}:\Theta_{\alpha}\rightarrow\mathbb{R}$ by the following formula
\begin{equation*}
a_{\alpha}(\theta)=\frac{1}{\mathrm{ln}(b_{\alpha}(\theta)+5)}.
\end{equation*}
Let us form a family measures $\mu_{\alpha}:=R(a_{\alpha},\Theta_{\alpha})$. Then, for every $\alpha\in\mathbb{R}$ we have $\mu_{\alpha}\in M_{0}(G)$ and also
\begin{equation*}
\mathrm{supp}\widehat{\mu_{\alpha}}=\{\theta\in\Gamma:\widehat{\mu_{\alpha}}\neq 0\}=\Omega(\Theta_{\alpha}).
\end{equation*}
Since the intersection of every two different sets from the collection $\{\Theta_{\alpha}\}_{\alpha\in\mathbb{R}}$ is finite, we have by Lemma $\ref{ban}$, $\#\Omega(\Theta_{\alpha_{1}})\cap\Omega (\Theta_{\alpha_{2}})<\infty$ for $\alpha_{1},\alpha_{2}\in\mathbb{R}$, $\alpha_{1}\neq\alpha_{2}$. On the other hand,
\begin{equation*}
\mathrm{supp}\widehat{\mu_{\alpha_{1}}\ast\mu_{\alpha_{2}}}=\mathrm{supp}\widehat{\mu_{\alpha_{1}}}\cap \mathrm{supp}\widehat{\mu_{\alpha_{2}}}=\Omega(\Theta_{\alpha_{1}})\cap\Omega(\Theta_{\alpha_{2}})
\end{equation*}
which shows that $\mu_{\alpha_{1}}\ast\mu_{\alpha_{2}}$ is a trigonometric polynomial. It is clear from our choice of a function $a_{\alpha}$ that for every $\alpha\in\mathbb{R}$ the assumptions of Theorems $\ref{rip}$,$\ref{ips}$ are satisfied and so $\sigma(\mu_{\alpha})=\{z\in\mathbb{C}:|z|\leq 1\}$. Let us take now $z_{0}\in\mathbb{C}$, $|z_{0}|<1$ and $z_{0}\notin\mathbb{R}$. Then, there exists an open neighborhood $V$ of $z_{0}$ in $\{z\in\mathbb{C}:|z|\leq 1\}$ which does not intersect the real line. Put $U_{\alpha}:=\widehat{\mu_{\alpha}}^{-1}(V)\subset \triangle(M_{0}(G))$. By the continuity of the Gelfand transforms of elements the set $U_{\alpha}$ is open in $\triangle(M_{0}(G))$. We are going to show that $U_{\alpha_{1}}\cap U_{\alpha_{2}}=\emptyset$ for $\alpha_{1}\neq\alpha_{2}$ which will finish the proof. First of all, $U_{\alpha}\cap\Gamma=\emptyset$ for every $\alpha\in\mathbb{R}$ because the Fourier-Stieltjes transform of $\mu_{\alpha}$ is real-valued for every $\alpha\in\mathbb{R}$ and $V$ does not intersect the real line. Suppose on the contrary that there exists $\varphi\in U_{\alpha_{1}}\cap U_{\alpha_{2}}$ for some $\alpha_{1},\alpha_{2}\in\mathbb{R}$, $\alpha_{1}\neq\alpha_{2}$. Then $\widehat{\mu_{\alpha_{1}}}(\varphi)\neq 0$ and $\widehat{\mu_{\alpha_{2}}}(\varphi)\neq 0$ which gives
\begin{equation}\label{nzero}
\varphi(\mu_{\alpha_{1}}\ast\mu_{\alpha_{2}})\neq 0.
\end{equation}
But $\varphi$ belongs to $\triangle(M_{0}(G))\setminus\Gamma$ and $\mu_{\alpha_{1}}\ast\mu_{\alpha_{2}}$ is a trigonometric polynomial and all trigonometric polynomials are in Zafran's ideal $\mathscr{C}$. Hence ($\ref{nzero}$) contradicts item 1. of Theorem $\ref{zz}$.
\end{proof}
\begin{rem}\label{upo}
In fact, an examination of the proof given above shows that we did not use the whole strength of Theorem \ref{zz}. It would be enough to refer to the following fact: if $\mu\in M(G)$ is an absolutely continous measure then the Gelfand tranform of $\mu$ vanishes off $\Gamma$, which is true for any locally compact abelian group (not necessary compact) because $L^{1}(G)$ forms the ideal in $M(G)$ (put $I=L^{1}(G)$ in the discussion following this remark).
\end{rem}
Let $A$ be a commutative, unital Banach algebra with closed ideal $I$. Then (see for example Lemma 2.2.5 in \cite{kan})
\begin{equation*}
\triangle(A)=\triangle(I)\cup h(I),\text{ }\triangle(I)\cap h(I)=\emptyset
\end{equation*}
where $h(I):=\{\varphi\in\triangle(A):\varphi|_{I}=0\}$ and $\triangle(I)$ is open in $\triangle(A)$. In our case, since $M_{0}(G)$ is a closed ideal in $M(G)$ we have the following corollary from Theorem $\ref{glo}$.
\begin{cor}\label{zww}
Let $G$ be a compact abelian group. Then $\triangle(M(G))$ contains continuum pairwise disjoint open subsets. In particular, $\triangle(M(G))$ is not separable.
\end{cor}
In the set-theoretical terminology we may say that the Gelfand space of a measure algebra on a compact group (also for non-compact as will be seen in Theorem \ref{ogw}) does not have ccc (countable chain condition).
\section{The general case}
The aim of this section is to prove the generalisation of Theorem \ref{glo}.
\begin{tw}\label{og}
Let $G$ be a non-discrete locally compact abelian group. Then $\triangle(M_{0}(G))$ contains continuum pairwise disjoint open sets. In particular, this space is not separable.
\end{tw}
Analogously as in the compact case (see the discussion after Remark \ref{upo} and Corollary \ref{zww}) we are able to use the last theorem to obtain the information on the topological properties of $\triangle(M(G))$.
\begin{cor}\label{ogw}
Let $G$ be a non-discrete locally compact abelian group. Then $\triangle(M(G))$ contains continuum pairwise disjoint open sets. In particular, this space is not separable.
\end{cor}
We start with two lemmas.
\begin{lem}\label{kon}
In order to prove Theorem \ref{og} it suffices to show that $M_{0}(G)$ contains continuum indepedent power Hermitian probability measures $\{\mu_{\alpha}\}_{\alpha\in\mathbb{R}}$ such that $\mu_{\alpha}\ast\mu_{\beta}\in L^{1}(G)$ for all $\alpha\neq\beta$.
\end{lem}
\begin{proof}
We will give two proofs.
\\
The first one is a simple observation that the assertion follows from an examination of the proof of Theorem \ref{glo} (instead of the use of Theorem \ref{zz} to obtain a final contradiction we base on the fact the Gelfand transform of an absolutely continuos measure vanishes off $\Gamma$ - see Remark \ref{upo}).
\\
The second proof is more direct: since each $\mu_{\alpha}$ is a Hermitian probability measure with independent powers we have $\sigma(\mu_{\alpha})=\{z\in\mathbb{C}:|z|\leq 1\}$ (see Theorem \ref{ips}). Let $z_{0}=\frac{1+i}{\sqrt{2}}$. For each $\alpha$, let $U_{\alpha}=\{\varphi\in\triangle(M_{0}(G)):|\widehat{\mu_{\alpha}}(\varphi)-z_{0}|<\frac{1}{10}\}$. We claim that $U_{\alpha}\cap U_{\beta}=\emptyset$ for $\alpha\neq\beta$. Suppose to the contrary that there exists $\varphi$ in this intersection. Then
\begin{equation*}
|\widehat{\mu_{\alpha}\ast\mu_{\beta}}(\varphi)-i|=|\widehat{\mu_{\alpha}}(\varphi)\widehat{\mu_{\beta}}(\varphi)-\widehat{\mu_{\alpha}}(\varphi)z_{0}+\widehat{\mu_{\alpha}}(\varphi)z_{0}-i|<\frac{2}{10}.
\end{equation*}
Hence, $\sigma(\mu_{\alpha}\ast\mu_{\beta})\cap \{z:|z-i|<\frac{2}{10}\}\neq\emptyset$. On the other hand, because $\mu_{\alpha}\ast\mu_{\beta}\in L^{1}(G)$ is a Hermitian probability measure, $\sigma(\mu_{\alpha}\ast\mu_{\beta})\subset(\widehat{\mu_{\alpha}}\widehat{\mu_{\beta}})(\Gamma)\cup\{0\}\subset [-1,1]$. This contradiction completes the proof.
\end{proof}
\begin{lem}\label{tech}
Let $G$ be a non-discrete locally compact abelian group. Then $M_{0}(G)$ contains continuum indepedent power Hermitian probability measures $\{\mu_{\alpha}\}_{\alpha\in\mathbb{R}}$ such that $\mu_{\alpha}\ast\mu_{\beta}\in L^{1}(G)$ for all $\alpha\neq\beta$ in each of the following cases:
\begin{enumerate}
  \item $G$ is compact.
  \item $G=\mathbb{R}$
  \item $G=\mathbb{R}^{n}\times H$, where $H$ is a locally compact group and $n\geq 1$.
  \item $G$ has an open, compact (infinite) subgroup $H$.
\end{enumerate}
\end{lem}
\begin{proof}
The first part was established in Theorem \ref{glo}.
\\
By the first part there exists a set $\{\nu_{\alpha}\}_{\alpha\in\mathbb{R}}$ of continuum independent power Hermitian probability measures in $M_{0}(\mathbb{T})$ such that the Fourier-Stieltjes transform of $\nu_{\alpha}\ast\nu_{\beta}$ has a finite support for each $\alpha\neq\beta$. For each $\alpha$ let $\mu_{\alpha}$ be the measure on $\mathbb{R}$ whose Fourier-Stieltjes transform agrees on $\mathbb{Z}\subset\mathbb{R}$ with $\nu_{\alpha}$ and is linear in the gaps. Then $\{\mu_{\alpha}\}_{\alpha\in\mathbb{R}}$ is a set of independent power Hermitian probability measures in $M_{0}(\mathbb{R})$ (see appendix A.7 in \cite{grmc}). We have to verify that $\mu_{\alpha}\ast\mu_{\beta}\in L^{1}(\mathbb{R})$ for $\alpha\neq\beta$. Let $L:M(\mathbb{R})\rightarrow M(\mathbb{T})$ be a mapping defined by the quotient homomorphism of groups, $P:\mathbb{R}\rightarrow \mathbb{R}/2\pi\mathbb{Z}=\mathbb{T}$. Then $L$ sends positive measures to positive measures and positive singular measures to singular measures. Let us justify the second fact. If a positive measure $\rho$ on the real line assigns non-zero value to a Lebesgue null Borel set $E\subset\mathbb{R}$ then $L(\rho)(P(E))=\rho(E)>0$. Of course, $P(E)$ is a Lebesgue null set in $\mathbb{T}$, so $L$ does indeed take non-absolutely continous measures to non-absolutely continous measures. On the level of Fourier-Stieltjes transforms we have $\widehat{L(\rho)}=\widehat{\rho}|_{\mathbb{Z}}$. Therefore, the probability measure $L(\mu_{\alpha}\ast\mu_{\beta})$ has the Fourier-Stieltjes transform $\widehat{L(\mu_{\alpha}\ast\mu_{\beta})}=\widehat{\mu_{\alpha}}\widehat{\nu_{\beta}}$ with a compact support. Thus the measure $\mu_{\alpha}\ast\mu_{\beta}$, which is a preimage of $\nu_{\alpha}\ast\nu_{\beta}$, is absolutely continous proving the second part of the lemma.
\\
If $n=1$ we take the products of each of the continuum measures from the second part with the Haar measure on a fixed, relatively compact, open subset of $H$. If $n>1$, we write $G=\mathbb{R}\times(\mathbb{R}^{n-1}\times H)$ and consider the products of each of the continuum measures from the second part with a Haar measure on a fixed, relatively compact, open subset of $\mathbb{R}^{n-1}\times H$.
\\
The first part applied to $H$ gives continuum measures on $H$ with the required properties (we treat those measures as elements in $M(H)$). The Fourier-Stieltjes transforms of those measures are constant on each translate of the anihilator of $H$ (the group of all elements $\gamma$ from $\Gamma$ satisfying $\gamma(x)=1$ for all $x\in H$) by Theorem 2.7.1 in \cite{r}. But the anihilator of $H$ is equal to the dual group of $G/H$ (check Theorem 2.1.2 in \cite{r}) which is a discrete group showing that the anihilator of $H$ is a compact group. By this considerations the measures $\mu_{\alpha}$ are in $M_{0}(G)$ and the other properties of these measures as elements of $M(G)$ are obvious.
\end{proof}
From the last lemma we have an immediate corollary.
\begin{cor}\label{mrr}
Let $G$ be a non-discrete locally compact abelian group. Then $M_{0}(G)$ contains continuum independent power Hermitian probability measures $\{\mu_{\alpha}\}_{\alpha\in\mathbb{R}}$ such that $\mu_{\alpha}\ast\mu_{\beta}\in L^{1}(G)$ for all $\alpha\neq\beta$.
\end{cor}
\begin{proof}
By the structure theorem for locally compact abelian groups (see for example \cite{r}), $G=\mathbb{R}^{n}\times H$ where $H$ has a compact open subgroup. If $n>0$ then the second part of Lemma \ref{tech} establishes the corollary. If $n=0$ then $G$ is either compact and we are able to apply the first part of Lemma \ref{tech} or $G$ has a compact open subgroup and we can use the fourth part.
\end{proof}
The proof of Theorem \ref{og} is a straightforward application of Corollary \ref{mrr} and Lemma \ref{kon}.
\section{An open problem}
Let us start with a definition.
\begin{de}\label{gen}
Let $A$ be a commutative Banach algebra with unit and let $S\subset\triangle(A)$. We say that $S$ generates the spectrum in $A$, if for every $x\in A$ we have
\begin{equation*}
\sigma(x)=\overline{\{\varphi(x):\varphi\in S\}}.
\end{equation*}
In addition, we say that the spectrum in $A$ is countably generated, if there exists a countable set $S\subset\triangle(A)$ which generates the spectrum in $A$.
\end{de}
Using Corollary \ref{ogw} we prove the following fact.
\begin{fakt}
Let $G$ be a non-discrete locally compact abelian group. Then the spectrum in $M(G)$ (also in $M_{0}(G)$) is not countably generated.
\end{fakt}
\begin{proof}
Let $S\subset\triangle(M(G))$ be any set which generates the spectrum in $M(G)$ and let us consider the uncountable family of measures $\{\mu_{\alpha}\}_{\alpha\in\mathbb{R}}\subset M(G)$ constructed in the proof of Theorem $\ref{ogw}$ together with an uncountable family of pairwise disjoint open sets $\{U_{\alpha}\}_{\alpha\in\mathbb{R}}\subset\triangle(M(G))$ used in the same proof. Then it is obvious that $S\cap U_{\alpha}\neq\emptyset$ for every $\alpha\in\mathbb{R}$ (since otherwise $\sigma(\mu_{\alpha})\neq\overline{\widehat{\mu_{\alpha}}(S)}$ contradicting the definition of $S$) and so the set $S$ has to be uncountable.
\end{proof}
It is clear that if a commutative, unital Banach algebra $A$ has separable Gelfand space then the spectrum in $A$ is countably generated. However, the reverse implication seems to be an open problem.
\begin{pro}
Let $A$ be commutative Banach algebra with unit and assume that the spectrum in $A$ is countably generated. Does it follow that $\triangle(A)$ is separable?
\end{pro}
Let us add two remarks concerning the last question.
\\
For the measure algebra on the compact group, proving that the Gelfand space is non-separable appears to be at the same level of difficulty as proving that the spectrum in this algebra is not countably generated. Nonetheless, it is not the case in general, which we can realize easily considering the algebra $H^{\infty}(\mathbb{D})$ of bounded holomorphic functions on the unit disc. In this example, the spectrum is countably generated and the proof of this fact is trivial but quite the opposite, the separability of $\triangle(H^{\infty}(\mathbb{D}))$ is quite close to the content of the famous Corona theorem of Carleson.
\\
It is worth-noting that the discussed problem is solved for some important classes of commutative Banach algebras. Indeed, let us take a commutative, unital Banach algebra $A$ satisfying $\mathfrak{M}(A)=\partial(A)$ (the last abbreviation stands for the Shilov boundary of $A$, see \cite{kan} or \cite{ż} for the definition) and assume that the spectrum in $A$ is generated by a countable set $S\subset\mathfrak{M}(A)$. Towards the contradiction suppose that $\overline{S}\neq\mathfrak{M}(A)$. Then, there exists $\varphi\in\mathfrak{M}(A)=\partial(A)$ and an open set $U\ni\varphi$ with the property $U\cap\overline{S}=\emptyset$. Recalling that the points in $\partial(A)$ have the following so-called peak point property (for a proof, consult \cite{kan} or \cite{ż}) i.e. there exists $x\in A$ and open subset $V\ni\varphi$ of $U$ such that $\|\widehat{x}|_{V}\|_{\infty}>\|\widehat{x}|_{X\setminus V}\|_{\infty}$, we obtain a contradiction because $\overline{S}\subset X\setminus U\subset X\setminus V$ and so $\|\widehat{x}|_{X\setminus U}\|_{\infty}=r(x)$. In particular, since every commutative, symmetric Banach$^{\ast}$ algebra $A$ satisfies $\mathfrak{M}(A)=\partial(A)$ and also every commutative, regular Banach algebra has this property, we have a broad collection of significant examples for which our problem has an affirmative answer. For arbitrary commutative Banach algebra the above argument shows that the Shilov boundary is included in the closure of any set which generates the spectrum in this algebra.

\end{document}